\documentclass[11pt]{amsart}
\usepackage[english,activeacute]{babel}
\usepackage{amsmath,amsfonts,amssymb,amstext,amsthm,amscd,mathrsfs,amsbsy,centernot}
\usepackage{xcolor}
\usepackage{hyperref}

\newtheorem{theorem}{Theorem}[section]
\newtheorem{proposition}[theorem]{Proposition}
\newtheorem{corollary}[theorem]{Corollary}
\newtheorem{lemma}[theorem]{Lemma}

\theoremstyle{definition}
\newtheorem{definition}[theorem]{Definition}
\newtheorem{example}[theorem]{Example}
\newtheorem{remark}[theorem]{Remark}

\newtheorem{theoremx}{Theorem}

\newcommand{\N}{\mathbb N}
\newcommand{\Z}{\mathbb Z}

\newcommand{\R}{\mathbb R}
\newcommand{\C}{\mathbb C}
\newcommand{\K}{\mathbb K}

\newcommand{\pol}{\mathbb \K[x_1,\ldots,x_n]}
\newcommand{\polmat}{\mathbb \K[x_{ij}|1\leq i\leq m,1\leq j\leq n]}
\newcommand{\lau}{\mathbb \K[t_1^{\pm},\ldots,t_d^{\pm}]}
\newcommand{\kg}{\K[t^{\Gamma}]}

\newcommand{\V}{\mathbf{V}}

\newcommand{\A}{\mathcal{A}}

\newcommand{\sd}{\check{\sigma}}

\newcommand{\Ga}{\Gamma}

\newcommand{\ga}{\gamma}

\newcommand{\conv}{\operatorname{Conv}}

\newcommand{\sing}{\operatorname{Sing}}
\newcommand{\ch}{\operatorname{char}}
\newcommand{\Gr}{\operatorname{Gr}}

\newcommand{\Nw}{\mathcal N(\Ga_p)}

\newcommand{\hi}{\mathcal{H}}

\newcommand{\jp}{\mathcal{J}_p}
\newcommand{\jo}{\mathcal{J}_0}

\newcommand{\ig}{I_{\Gamma}}

\newcommand{\xg}{X_{\Gamma}}

\newcommand{\mat}{M_{m,n}}
\newcommand{\mt}{M_{m,n}^t}
\newcommand{\md}{M_{m,n}^2}

\newcommand{\lmn}{L_{m,n}}
\newcommand{\tlmu}{\tilde{L}_{m,n-1}}
\newcommand{\lmu}{L_{m,n-1}}
\newcommand{\lmd}{L_{m,2}}

\newcommand{\mb}{\mbox{ }}

\begin{document}

\title[Nash blowups 2-determinantal varieties]{Nash blowups of 2-generic determinantal varieties in positive characteristic}


\author{Tha\'is M. Dalbelo}
\author{Daniel Duarte}
\author{Maria Aparecida Soares Ruas}

\thanks{T. M. Dalbelo was supported by FAPESP-Grant 2019/21181-0 and by CNPq grant 403959/2023-3.}
\thanks{D. Duarte was supported by CONAHCYT project CF-2023-G-33 and PAPIIT grant IN117523.}
\thanks{M. A. S. Ruas was supported by FAPESP-Grant 2019/21181-0 and by CNPq grant “Bolsa de Produtividade em Pesquisa” 305695/2019-3.}

\address{Tha\'is M. Dalbelo, Universidade Federal de Sao Carlos}
\email{thaisdalbelo@ufscar.br}

\address{Daniel Duarte, Centro de Ciencias Matem\'aticas, UNAM}
\email{adduarte@matmor.unam.mx}

\address{Maria Aparecida Soares Ruas, Instituto de Ciências Matem\'aticas e de Computação, Universidade de São Paulo}
\email{maasruas@icmc.usp.br}

\subjclass[2020]{14E15,14M12,14M25}
\keywords{Nash blowup, generic determinantal variety, toric variety, positive characteristic fields}

\date{\today}

\dedicatory{To Mark Spivakovsky, on the occasion of his 64th anniversary.}

\begin{abstract}
We show that the Nash blowup of 2-generic determinantal varieties over fields of positive characteristic is non-singular. We prove this in two steps. Firstly, we explicitly describe the toric structure of such varieties. Secondly, we show that in this case the combinatorics of Nash blowups are free of characteristic. The result then follows from the analogous result in characteristic zero proved by W. Ebeling and S. M. Gusein-Zade.
\end{abstract}

\maketitle



\section*{Introduction}

The Nash blowup of an algebraic variety is a modification that replaces singular points by the tangency at nearby smooth points. It has been proposed to resolve singularities by iterating this process \cite{S,No}. This question has been extensively studied over fields of characteristic zero \cite{No,R,GS1,GS2,Hi,Sp,GS3,EGZ,At,GM,GT,D1,Cha,DG}.

In the case of positive characteristic fields, the study of Nash blowups was discouraged by an example of A. Nobile in which the Nash blowup turns out to be isomorphic to the initial singular variety \cite{No}. This undesired behaviour was overcome by adding the condition of normality. More precisely, the Nash blowup of a normal and singular variety over fields of positive characteristic is not an isomorphism \cite{DN1}. That result was the starting point for a renewed interest in the study of Nash blowups in positive characteristic. The first case that was explored in this context was the case of toric varieties.

The study of Nash blowups of toric varieties over fields of characteristic zero was initiated by G. Gonz\'alez-Sprinberg. He gave a combinatorial description of the Nash blowup of a toric variety by using the so-called logarithmic Jacobian ideal. Using that description, Gonz\'alez-Sprinberg showed that iterated normalized Nash blowups gives a resolution of singularities of normal toric surfaces \cite{GS1}. 

The positive characteristic versions of Gonz\'alez-Sprinberg's results were recently obtained by D. Duarte, J. Jeffries, and L. N\'{u}\~nez-Betancourt \cite{DJNB}. Firstly, the corresponding combinatorial description followed by means of an analogous of the logarithmic Jacobian ideal that takes into account the characteristic. Secondly, it was shown that the combinatorial object describing the Nash blowup of a normal toric surface is the same regardless of the characteristic. We refer to this fact by saying that \textit{the combinatorics of Nash blowups are free of characteristic}. As a consequence, iterated normalized Nash blowups gives a resolution of normal toric surfaces in positive characteristic as well.

In a different direction, the work of W. Ebeling and S. M. Gusein-Zade shows that the Nash blowup of a generic determinantal variety is non-singular in characteristic zero \cite{EGZ}. Generic determinantal varieties are defined in spaces of matrices in terms of vanishing of minors. This family of varieties has been extensively studied (see, for instance, \cite{ACGH,BV,EH}).

In this paper we show the positive characteristic version of W. Ebeling and S. M. Gusein-Zade's result in the case of 2-generic determinantal varieties. Our interest in this particular case comes from the fact that 2-generic determinantal varieties are also toric varieties. Hence we can use the combinatorial tools previously discussed.

The proof of our main theorem is divided in two main steps. We first describe in toric terms the 2-generic determinantal varieties. For that goal we introduce an explicit set of vectors whose associated toric varieties are the 2-generic determinantal varieties. 

\begin{theoremx}[see Theorem \ref{M2 toric}]\label{theorem 1}
Let $m,n\in\N$, $m,n\geq2$. Denote as $\md$ the corresponding 2-generic determinantal variety. Let $\Ga\subset\Z^{m+n-1}$ be the semigroup generated by the following set of vectors:
\begin{align}
\A=\{e_1&,e_2,\ldots,e_m,e_{m+1},e_{m+2},\ldots,e_{m+n-1},\notag\\
-&e_1+e_2+e_{m+1},-e_1+e_3+e_{m+1},\ldots,-e_1+e_m+e_{m+1},\notag\\
-&e_1+e_2+e_{m+2},-e_1+e_3+e_{m+2},\ldots,-e_1+e_m+e_{m+2},\notag\\
&\vdots\notag\\
-&e_1+e_2+e_{m+n-1},-e_1+e_3+e_{m+n-1},\ldots,-e_1+e_m+e_{m+n-1}\}.\notag
\end{align}
Then $\md$ coincides with the toric variety defined by $\Ga$.
\end{theoremx}

With this description at hand we show that, like in the case of normal toric surfaces, the combinatorics of Nash blowups of 2-generic determinantal varieties are free of characteristic. The following theorem is then a consequence of W. Ebeling and S. M. Gusein-Zade's result. It is important to emphasize that our result does not require to apply a normalized Nash blowup, as in the case of normal toric surfaces.

\begin{theoremx}[see Theorem \ref{Nash M2 smooth}]\label{theorem 2}
Assume that $\ch(\K)>0$. The Nash blowup of $\md$ is non-singular.
\end{theoremx}

The paper is divided as follows. In the first section we recall the basics of toric and generic determinantal varieties. We also prove Theorem \ref{theorem 1}. Section 2 is devoted to explain the combinatorial descripcion of Nash blowups of toric varieties over zero and prime characteristic fields. Finally, we prove Theorem \ref{theorem 2} in Section 3.



\section{Toric and generic determinantal varieties}

Throughout this note $\K$ denotes an algebraically closed field of arbitrary characteristic.
\\

Let us start by recalling the definition of a toric variety.

\begin{definition}\cite{St}\label{def toric}
Let $\Gamma\subseteq\Z^d$ be a semigroup generated by a finite set of vectors $\A=\{\gamma_1,\ldots,\gamma_n\}$. Consider the $\K$-algebra homomorphism $\pi_{\Gamma}:\pol\to\lau$, $x_i\mapsto t^{\gamma_i}$. Let $\ig=\ker\pi_{\Gamma}$. The variety $\xg=\V(\ig)\subseteq\K^n$ is called the toric variety defined by $\Gamma$. We denote as $\kg$ the image of $\pi_{\Gamma}$, which is the coordinate ring of $\xg$.
\end{definition}

The following are basic properties of toric varieties that we will use.

\begin{proposition}\label{prop toric}\cite[Chapter 13]{St}
Consider the notation of Definition \ref{def toric}.
\begin{enumerate}
\item $\xg$ is irreducible.
\item If $\Z\A=\Z^d$ then $\xg$ has dimension $d$.
\item $\xg$ is a normal variety if and only if $\Ga=\R_{\geq0}\A\cap\Z^d$.
\end{enumerate}
\end{proposition}

Now we recall the definition of a generic determinantal variety.

\begin{definition}\cite{ACGH}\label{gen det}
Let $\mat$ be the $\K$-vector space of $(m\times n)$-matrices with entries in $\K$. Let $\mt$ be the subset of $\mat$ consisting of matrices of rank less than $t$, that is, the matrices all of whose $(t\times t)$-minors vanish. $\mt$ is called a \textit{generic determinantal variety}. 

In other words, let $L=(x_{ij})_{1\leq i\leq m,1\leq j\leq n}$, where $x_{ij}$ are indeterminates, and denote as $J_t\subset\polmat$ the ideal generated by all $(t\times t)$-minors of $L$. Then $\mt=\V(J_t)\subset\K^{mn}$.
\end{definition}

The following are basic properties of generic determinantal varieties that we will use.

\begin{proposition}\label{prop det}
Consider the notation of Definition \ref{gen det}.
\begin{enumerate}
\item $J_t$ is a prime ideal. In particular $\mt$ is an irreducible variety.
\item The dimension of $\mt$ is $mn-(m-t+1)(n-t+1)$.
\item $\mt$ is a normal Cohen-Macaulay variety.
\end{enumerate}
\end{proposition}

We refer to \cite{BV, EH} for properties of determinantal varieties over commutative rings.

We are interested in studying the case $t=2$. One special feature of this particular case is the fact that $\md$ is a binomial variety, that is, its defining equations are binomials. Our first goal is to prove that $\md$ is actually a normal toric variety. We do this by describing explicitly its corresponding semigroup. Before going into this discussion, we present two examples.

\begin{example}\label{exam 1}
Let $M^2_{2,2}=\V(x_1x_4-x_2x_3)\subset\K^4$. In this case it is known that $M^2_{2,2}=\xg$, where $\Ga$ is the semigroup generated by the set of vectors $\{e_1,e_2,e_3,-e_1+e_2+e_3\}\subset\Z^3$ \cite[Example 1.1.18]{CLS}. Here $\{e_1,e_2,e_3\}$ denotes the canonical basis of $\Z^3$.
\end{example}

\begin{example}\label{exam 2}
Consider $M^2_{2,3}$. Let $f_1=x_1x_4-x_2x_3$, $f_2=x_1x_6-x_2x_5$, and $f_3=x_3x_6-x_4x_5$, which are the $(2\times 2)$-minors of the following matrix:
\[
\begin{pmatrix}
x_1& 	x_3& 	x_5\\
x_2&  	x_4&   x_6
\end{pmatrix}.
\]
By definition, $M^2_{2,3}=\V(f_1,f_2,f_3)\subset\K^6$. By Proposition \ref{prop det}, $\dim M^2_{2,3}=4$. Let $\Ga\subset\Z^4$ be the semigroup generated by $\A=\{\ga_1,\ldots,\ga_6\}$, where
\begin{align}
\ga_1&=e_1,\notag\\
\ga_2&=e_2,\notag\\
\ga_3&=e_3,\notag\\
\ga_4&=-e_1+e_2+e_3,\notag\\
\ga_5&=e_4,\notag\\
\ga_6&=-e_1+e_2+e_4.\notag
\end{align}
We show that $M^2_{2,3}=\xg$. Let $\ig\subset\K[x_1,\ldots,x_6]$ be the corresponding toric ideal. Since $\Z\A=\Z^4$, Proposition \ref{prop toric} gives $\dim \xg=4$. A straightforward computation shows that $f_1,f_2,f_3\in\ig$. Hence, $\xg\subset M^2_{2,3}\subset\K^6$. These varieties are irreducible and have the same dimension. We conclude that $M^2_{2,3}=\xg$.
\end{example}

The previous examples are particular cases of the following result.

\begin{theorem}\label{M2 toric}
Let $m,n\in\N$, $m,n\geq2$. Let $\Ga\subset\Z^{m+n-1}$ be the semigroup generated by the following set of vectors:
\begin{align}
\A=\{e_1&,e_2,\ldots,e_m,e_{m+1},e_{m+2},\ldots,e_{m+n-1},\notag\\
-&e_1+e_2+e_{m+1},-e_1+e_3+e_{m+1},\ldots,-e_1+e_m+e_{m+1},\notag\\
-&e_1+e_2+e_{m+2},-e_1+e_3+e_{m+2},\ldots,-e_1+e_m+e_{m+2},\notag\\
&\vdots\notag\\
-&e_1+e_2+e_{m+n-1},-e_1+e_3+e_{m+n-1},\ldots,-e_1+e_m+e_{m+n-1}\}.\notag
\end{align}
Then $\md=\xg$.
\end{theorem}
\begin{proof}
Denote the elements of $\A$ as follows:
\begin{align}
\ga_{i1}&=e_i, \mbox{ for }i\in\{1,\ldots,m\},\notag\\
\ga_{1j}&=e_{m+j-1} \mbox{ for } j\in\{2,\ldots,n\},\notag\\
\ga_{ij}&=-e_1+e_{i}+e_{m+j-1}\mbox{ for } i\in\{2,\ldots,m\} \mbox{ and } j\in\{2,\ldots,n\}.\notag
\end{align}
Recall that $J_2\subset\polmat$ denotes the ideal generated by the $(2\times 2)$-minors of the matrix $(x_{ij})_{1\leq i\leq m,1\leq j\leq n}$. Associate to each indeterminate $x_{ij}$ the vector $\ga_{ij}$ . We claim that $J_2\subset \ig$. 

To prove the claim we show that the multiplicative relation $x_{ij}x_{lk}=x_{ik}x_{lj}$, $1\leq i<l\leq m$ and $1\leq j<k\leq n$, corresponds to the additive relation $\ga_{ij}+\ga_{lk}=\ga_{ik}+\ga_{lj}$. We prove this by considering cases.
\begin{enumerate}
\item Let $i=j=1$. Then $\ga_{11}+\ga_{lk}=e_1+e_l+e_{m+k-1}-e_1$ and $\ga_{1k}+\ga_{l1}=e_{m+k-1}+e_l$.
\item Let $i\geq 2, j=1$. Then $\ga_{i1}+\ga_{lk}=-e_1+e_i+e_l+e_{m+k-1}$ and $\ga_{ik}+\ga_{l1}=-e_1+e_i+e_{m+k-1}+e_l$.
\item Let $i=1, j\geq2$. Then $\ga_{1j}+\ga_{lk}=-e_1+e_{m+j-1}+e_l+e_{m+k-1}$ and $\ga_{1k}+\ga_{lj}=-e_1+e_{m+k-1}+e_l+e_{m+j-1}$.
\item Let $i\geq 2, j\geq2$. Then $\ga_{ij}+\ga_{lk}=-e_1+e_i+e_{m+j-1}-e_1+e_l+e_{m+k-1}$ and $\ga_{ik}+\ga_{lj}=-e_1+e_i+e_{m+k-1}-e_1+e_l+e_{m+j-1}$.
\end{enumerate}
This shows the claim. Therefore, we have $\xg=\V(\ig)\subset\V(J_2)\subset \K^{mn}$. Since $\Z\A=\Z^{m+n-1}$ we have that $\dim \xg=m+n-1$, by Proposition \ref{prop toric}. On the other hand, $\dim \md=m+n-1$ by Proposition \ref{prop det}. Since these two varieties are irreducible, we conclude that $\xg=\md$.
\end{proof}

\begin{corollary}
$\xg$ is a normal variety. In particular, $\Ga$ is a saturated semigroup, i.e., $\Ga=\R_{\geq0}\A\cap\Z^{m+n-1}$. Moreover, $\R_{\geq0}\A$ is a strongly convex cone.
\end{corollary}
\begin{proof}
The first part follows from Theorem \ref{M2 toric} and Propositions \ref{prop toric} and \ref{prop det}. The second part is a consequence of the fact $0\in\md=\xg$.
\end{proof}

\begin{corollary}
$\A$ is the minimal generating set of $\Ga$.
\end{corollary}
\begin{proof}
Let $\sigma\subset\R^{m+n-1}$ denote the dual cone of $\R_{\geq0}\A$. Notice that $\sigma$ is full-dimensional and strongly convex since the cone $\R_{\geq0}\A$ satisfy these properties. From the previous corollary we know that $\xg$ is the normal toric variety corresponding to $\sigma$. 

Let $\hi$ be the minimal generating set of $\Ga$. Denote as $T_0\xg\subset\K^{mn}$ the tangent space of $\xg$ at the origin. Then $\dim T_0\xg=|\hi|$ \cite[Lemma 1.3.10]{CLS}. We claim that $T_0\xg=\K^{mn}$. Assume the claim for the moment. Then $|\hi|=mn$. Since $\hi\subset\A$ and $|\A|=mn$ we conclude $\hi=\A$.

To prove the claim recall that $J_2$ and $\ig$ are prime ideals and $\K$ is algebraically closed. Hence, $J_2=\mathbb{I}(\md)=\mathbb{I}(\xg)=\ig$. Recall that the generators of $J_2$ are of the form $x_{ij}x_{kl}-x_{il}x_{kj}$. This implies that the Jacobian matrix of $J_2$, evaluated at $0$, is the zero matrix. In particular, its kernel is $\K^{mn}$. This proves the claim.

\end{proof}

\begin{remark}
With the previous notation, computational experimentation suggests that the dual cone of $\R_{\geq0}\A$ has as generators the following set of vectors in $\Z^{m+n-1}$:
$$\{e_2,\ldots,e_{m+n-1},e_1+e_2+\ldots+e_m, e_1+e_{m+1}+\cdots+e_{m+n-1}\}.$$
\end{remark}


\section{Nash blowups of toric varieties}

In this section we recall the definition of Nash blowups and its combinatorial description in the case of toric varieties defined over fields of arbitrary characteristic.

\begin{definition}
Let $\K$ be an algebraically closed field of arbitrary characteristic. Let $X\subseteq\K^n$ be an equidimensional algebraic variety of dimension~$d$. Consider the Gauss map:
\begin{align}
G:X\setminus\sing(X)&\rightarrow\Gr(d,n)\notag\\
x&\mapsto T_xX,\notag
\end{align}
where $\Gr(d,n)$ is the Grassmanian of $d$-dimensional vector spaces in $\K^n$, and $T_xX$ is the tangent space of $X$ at $x$. Denote by $X^*$ the Zariski closure of the graph of $G$. Call $\nu$ the restriction to $X^*$ of the projection of $X\times\Gr(d,n)$ to $X$. The pair $(X^*,\nu)$ is called the Nash blowup of $X$.
\end{definition}

The first step towards a combinatorial description of Nash blowups of toric varieties in characteristic zero is given by the so-called logarithmic Jacobian ideal. This ideal was originally introduced by G. Gonz\'alez-Sprinberg \cite[Section 2]{GS1}, and was later revisited by several authors \cite{LJ-R,GT,ChDG}.

\begin{definition}
Suppose that $\ch(\K)=0$. Let $\Gamma\subset\Z^d$ be a semigroup generated by $\A=\{\gamma_1,\ldots,\gamma_n\}$. Assume that 
$\Z\A=\Z^d$. Consider the following ideal:
$$\jo=\langle t^{\gamma_{i_1}+\cdots+\gamma_{i_d}}\,|\,\det(\gamma_{i_1}\cdots\gamma_{i_d})\neq0, \, 1\leq i_1<\cdots<i_d\leq n \rangle\subseteq\kg.$$
The ideal $\jo$ is called the \textit{logarithmic Jacobian ideal of $\xg$}.
\end{definition}

\begin{theorem}[{\cite{GS1,LJ-R,GT}}]\label{Nash=log-0}
Suppose that  $\ch(\K)=0$. The Nash blowup of $\xg$ is isomorphic to the blowup of its logarithmic Jacobian ideal.
\end{theorem}

It is known that the previous theorem is false over fields of prime characteristic. 

\begin{example}\label{counterex}
Suppose that  $\ch(\K)=2$, $\Gamma=\N(\{2,3\})$. Then $X_{\Gamma}={\V(x^3-y^2)}$, $Bl_{\langle t^2,t^3\rangle}\xg$ is nonsingular but $\xg^*\cong\xg$ \cite[Example 1]{No}.
\end{example}

Theorem \ref{Nash=log-0} was recently generalized to positive characteristic fields as follows \cite{DJNB}.

\begin{definition}\label{log jac mod p}
Suppose that $\ch(\K)=p>0$. Let $\Gamma\subset\Z^d$ be a semigroup generated by $\A=\{\gamma_1,\ldots,\gamma_n\}$. Assume that 
$\Z\A=\Z^d$. Consider the following ideal:
$$\jp=\langle t^{\gamma_{i_1}+\cdots+\gamma_{i_d}}|\det(\gamma_{i_1}\cdots\gamma_{i_d})\neq0\, \mathrm{mod} \, p, 1\leq i_1<\cdots<i_d\leq n \rangle\subseteq\kg.$$
The ideal $\jp$ is called the \textit{logarithmic Jacobian ideal modulo $p$ of $\xg$}. 
\end{definition}

\begin{example}
Let $\Gamma=\N(\{2,3\})\subseteq\N$. Then $\mathcal{J}_2=\langle t^3 \rangle$, $\mathcal{J}_3=\langle t^2 \rangle$, and $\mathcal{J}_p=\langle t^2,t^3 \rangle$, for $p=0$ and $p\geq 5$.
\end{example}

\begin{theorem}\cite[Theorem 1.9]{DJNB}\label{Nash=log}
Suppose that  $\ch(\K)=p>0$. The Nash blowup of $\xg$ is isomorphic to the blowup of its logarithmic Jacobian ideal modulo $p$.
\end{theorem}

Using Theorems \ref{Nash=log-0} and \ref{Nash=log}, a combinatorial description of the Nash blowup of a toric variety can be obtained with the framework developed by P. Gonz\'alez and B. Teissier for the blowup of a toric variety along any monomial ideal \cite[Section 2.6]{GT}. We state the following theorem for the particular case of the blowup of the logarithmic Jacobian ideal modulo $p$, $p\geq0$.

\begin{theorem}\label{GT}\cite[Proposition 32]{GT}
Let $\Gamma\subset\Z^d$ be a semigroup generated by $\A=\{\gamma_1,\ldots,\gamma_n\}$ and such that $\Z\A=\Z^d$. Let $p$ denote $0$ or a prime number. Consider:
\begin{itemize}
\item $\sd:=\R_{\geq0}\Ga\subset\R^d$. Assume that $\sd$ is strongly convex.
\item $\Ga_p:=\{\gamma_{i_1}+\cdots+\gamma_{i_d}|\det(\gamma_{i_1} \cdots \gamma_{i_d})\neq0\mod p\}_{1\leq i_1 <\cdots < i_d \leq n}\subset\Z^d$.
\item $\Nw:=\conv\{m+\sd|m\in\Ga_p\}\subset\R^d$, where $\conv(\cdot)$ denotes the convex hull.
\item For each vertex $m_0$ of $\Nw$, let $\A_{m_0}=\A\cup\{m-m_0|m\in\Ga_p\}$. Denote as $\Ga_{m_0}$ the semigroup generated by $A_{m_0}$.
\end{itemize}
Then the set $\{X_{\Ga_m}|m\mbox{ is a vertex of }\Nw\}$ is an affine cover of the blowup of $\xg$ along the logarithmic Jacobian ideal modulo p. 
\end{theorem}

\begin{remark}\label{crit sing}
Consider the notation of the previous theorem. As a consequence of Theorems \ref{Nash=log-0}, \ref{Nash=log}, and \ref{GT}, the Nash blowup of $\xg$ is non-singular if and only if each $X_{\Gamma_m}$ is non-singular if and only if each $\Ga_m$ can be generated by $d$ of its elements.
\end{remark}

The previous theorem gives a combinatorial description of the affine charts of the Nash blowup of a toric variety. Let us apply this process to Example \ref{exam 1}.

\begin{example}\label{exam Nash 1}
Let $\A=\{\ga_1,\ldots,\ga_4\}\subset\Z^3$, where $\ga_1=e_1$, $\ga_2=e_2$, $\ga_3=e_3$, $\ga_4=-e_1+e_2+e_3$. Let $\Ga$ be the semigroup generated by $\A$. Let us write these vectors as columns of a matrix:
$$
L_{2,2}=\begin{pmatrix}
1&0&0&-1 \\
0&1&0&1	\\
0&0&1&1
\end{pmatrix}.
$$
The determinant of every maximal square submatrix of $L_{2,2}$ is $1$ or $-1$. This implies that $\Ga_p=\{(1,1,1),(0,2,1),(0,1,2),(-1,2,2)\},$ for $p=0$ or $p$ any prime number. 

Using Macaulay2 \cite{M2} to compute convex hulls and vertices, it can be shown that the semigroups $\Ga_{m_0}$ can be generated by three of its elements, for each vertex $m_0$ of $\Nw$. Hence each of these affine toric varieties is non-singular. Consequently, the Nash blowup of $\xg$ is non-singular over a field of any characteristic.
\end{example}


\section{Nash blowup of $\md$ in positive characteristic}

In this section we prove our second main theorem: the Nash blowup of $\md$ is non-singular over fields of positive characteristic. The first important ingredient toward that goal is the following theorem in characteristic zero.

\begin{theorem}\label{EGZ}\cite[Section 1]{EGZ}
Let $\K=\C$. The Nash blowup of $\mt$ is non-singular.
\end{theorem}

To give the positive characteristic version of this theorem in the case $t=2$, we study the maximal minors of the matrix whose columns are the vectors of $\A$ from Theorem \ref{M2 toric}. To describe the matrix we introduce some notation. 

Firstly, let $e_{1},\ldots,e_{m+n-1}$ denote the canonical basis of $\Z^{m+n-1}$. Given $j\in\{1,\ldots,n-1\}$, denote
$$B_{m+j}=\{-e_1+e_2+e_{m+j},-e_1+e_3+e_{m+j},\ldots,-e_1+e_m+e_{m+j}\}.$$
Let $m,n\geq2$. Abusing notation, we define the following matrix:
$$\lmn=(e_1\mb e_2\mb \cdots \mb e_{m+n-1}\mb B_{m+1}\mb B_{m+2}\mb\cdots\mb B_{m+n-1}).$$
Notice that $\lmn$ is a $(m+n-1\times mn)$-matrix.

\begin{example}
$$
L_{2,2}=\begin{pmatrix}
1&0&0&-1 \\
0&1&0&1	\\
0&0&1&1
\end{pmatrix}.
$$
$$
L_{2,3}=\begin{pmatrix}
1&0&0&0&-1&-1 \\
0&1&0&0&1&1\\
0&0&1&0&1&0\\
0&0&0&1&0&1
\end{pmatrix}.
$$
$$
L_{3,3}=\begin{pmatrix}
1&0&0&0&0&-1&-1&-1&-1 \\
0&1&0&0&0&1&0&1&0\\
0&0&1&0&0&0&1&0&1\\
0&0&0&1&0&1&1&0&0\\
0&0&0&0&1&0&0&1&1
\end{pmatrix}.
$$
\end{example}

\begin{theorem}\label{Nash M2 smooth}
Assume that $\ch(\K)=p>0$. The Nash blowup of $\md$ is non-singular.
\end{theorem}
\begin{proof}
Let $L_{m,n}$ be the matrix previously defined. In Proposition \ref{minors} we show that all $(m+n-1\times m+n-1)$-minors of $L_{m,n}$ are $0$, $1$ or $-1$. Thus, in the notation of Theorem \ref{GT}, $\Ga_0=\Ga_p$ for every prime number $p$. As a consequence, $\mathcal{N}(\Ga_0)=\mathcal{N}(\Ga_p)$ for every prime $p$, their vertices coincide, and hence the semigroups $\Ga_m$ are also independent of the characteristic. By Theorem \ref{EGZ} and Remark \ref{crit sing}, we conclude that the Nash blowup of $\md$ is non-singular also over fields of positive characteristic.
\end{proof}

Now we prove the claim made in the previous proof regarding the maximal minors of $\lmn$. 

\begin{lemma}
All $(m+1)$-minors of $L_{m,2}$ are $0$, $1$, or $-1$.
\end{lemma}
\begin{proof}
Firstly, notice that 
\begin{align}
\lmd&=(e_1\mb\cdots e_{m+1}\mb B_{m+1})\notag\\
&=(e_1\mb\cdots e_{m+1}\mb-e_1+e_2+e_{m+1}\mb\cdots-e_1+e_m+e_{m+1}).\notag
\end{align} 
Let $L=(v_1\,v_2\,\cdots\,v_{m+1})$ be a maximal square submatrix of $\lmd$. Consider the following cases, depending on the choice of $v_i's$ from $\{e_1,\ldots,e_m\}$, $\{e_{m+1}\}$ or $B_{m+1}$. We consider three main cases.
\begin{itemize}
\item[(1)] $\{v_1,\ldots,v_{m+1}\}=\{e_1,\ldots,e_{m+1}\}$. 
\item[(2)] One $v_i$ is $e_{m+1}$ and at least one $v_j$ belongs to $B_{m+1}$.
\item[(3)] All $v_i's$ belong to $\{e_1,\ldots,e_m\}\cup B_{m+1}$.
\end{itemize}

We prove that in each case the determinant of $L$ is $0$, $1$ or $-1$.
\begin{itemize}
\item[(1)] In this case all $v_i's$ are canonical basis elements. Then $\det L$ is $1$ or $-1$.
\item[(2)] After renumbering the vectors if necessary, assume 
\begin{align}
&v_1=e_{m+1},\notag\\
&v_2,\ldots,v_l\in B_{m+1} \mbox{ for some } 2\leq l \leq m, \notag\\ 
&v_{l+1},\ldots,v_{m+1}\in\{e_1,\ldots,e_m\}.\notag
\end{align} 
Let $2\leq i_2<\cdots< i_l\leq m$ be such that $v_k=-e_1+e_{i_k}+e_{m+1}$, for each $k\in\{2,\ldots,l\}$. Since $v_1=e_{m+1}$, we can replace the column $v_k$ by $v_k-v_1=-e_1+e_{i_k}$ for each $2\leq k\leq l$, without affecting the determinant of $L$ (up to sign). Hence, we want to compute the determinant of
$$L'=(e_{m+1}\mb -e_1+e_{i_2}\mb\cdots\mb-e_1+e_{i_l}\mb v_{l+1}\mb\cdots\mb v_{m+1}).$$
\end{itemize}
\begin{itemize}
\item[(2.1)] Suppose $v_j=e_1$ for some $l+1\leq j \leq m+1$. Replacing the columns $-e_1+e_{i_k}$ by $-e_1+e_{i_k}+v_j=e_{i_k}$, the matrix $L'$ becomes
$$L''=(e_{m+1}\mb e_{i_2}\mb\cdots\mb e_{i_l}\mb v_{l+1}\mb\cdots\mb v_{m+1}).$$
Notice that each column of $L''$ is a canonical basis element. Hence $\det L$ is $0$, $1$, or $-1$.
\item[(2.2)] Suppose $v_j\in\{e_{i_2},\ldots,e_{i_l}\}$ for some $l+1\leq j \leq m+1$. By simplicity of notation assume that $v_{l+1}=e_{i_2}$. In $L'$, replace the colum $-e_1+e_{i_2}$ by $-e_1+e_{i_2}-v_{l+1}=-e_1$. We obtain the following matrix
$$L''=(e_{m+1}\mb -e_1\mb -e_1+e_{i_3}\mb\cdots\mb-e_1+e_{i_l}\mb v_{l+1}\mb\cdots\mb v_{m+1}).$$
Having $-e_1$ as a column, now we can proceed as in $(2.1)$ to turn the columns $-e_1+e_{i_k}$ into $e_{i_k}$. As before, $\det L$ is $0$, $1$, or $-1$.
\item[(2.3)] Suppose $\{v_{l+1},\ldots,v_{m+1}\}\subset \{e_1,\ldots,e_m\}\setminus\{e_1,e_{i_2},\ldots,e_{i_l}\}$. By the cardinality of these sets, we conclude that some $v_i's$ are repeated. Hence $\det L=0$.
\end{itemize}
\begin{itemize}
\item[(3)] After renumbering the vectors if necessary, assume 
\begin{align}
&v_2,\ldots,v_l\in B_{m+1} \mbox{ for some } 2\leq l \leq m, \notag\\ 
&v_1,v_{l+1},\ldots,v_{m+1}\in\{e_1,\ldots,e_m\}.\notag
\end{align} 
Let $2\leq i_2<\cdots< i_l\leq m$ be such that $v_k=-e_1+e_{i_k}+e_{m+1}$, for each $k\in\{2,\ldots,l\}$.
\end{itemize}
\begin{itemize}
\item[(3.1)] Suppose $v_j\in\{e_{i_2},\ldots,e_{i_l}\}$ for some $l+1\leq j \leq m+1$ or $j=1$. By simplicity of notation assume that $v_1=e_{i_2}$. In $L$, replace the colum $v_2$ by $v_2-v_1=-e_1+e_{m+1}$. We obtain the following matrix
$$L'=(e_{i_2}\mb -e_1+e_{m+1}\mb v_3\mb\cdots v_l\mb v_{l+1}\mb\cdots\mb v_{m+1}).$$
Now replace each $v_{k}$ by $v_k-(-e_1+e_{m+1})=e_{i_k}$, for each $k\in\{3,\ldots,l\}$. Now permute the columns $e_{i_2}$ and $-e_1+e_{m+1}$. We obtain the following matrix
$$L''=(-e_1+e_{m+1}\mb e_{i_2}\mb e_{i_3}\mb\cdots\mb e_{i_l}\mb v_{l+1}\mb\cdots\mb v_{m+1}).$$
This matrix has the following shape
$$
\begin{pmatrix}
*	&  A \\
1 	& 0	
\end{pmatrix},
$$
where $A$ is a $(m\times m)$-matrix whose columns are canonical basis elements of $\Z^{m}$. Hence $\det L$ is $0$, $1$ or $-1$.
\end{itemize}
\begin{itemize}
\item[(3.2)] Suppose $\{v_1,v_{l+1},\ldots,v_{m+1}\}\subset \{e_1,\ldots,e_m\}\setminus\{e_{i_2},\ldots,e_{i_l}\}$. By the cardinality of these sets, we conclude that some $v_i's$ are repeated. Hence $\det L=0$.
\end{itemize}
\end{proof}


\begin{proposition}\label{minors}
Let $m,n\geq2$. All maximal minors of $\lmn$ are $0$, $1$, or $-1$.
\end{proposition}
\begin{proof}
We proceed by induction on $n\geq2$, the previous lemma being the case $n=2$. Recall that
$$\lmn=(e_1\mb e_2\mb \cdots \mb e_{m+n-1}\mb B_{m+1}\mb B_{m+2}\mb\cdots\mb B_{m+n-1}).$$
By definition, all columns of $\lmn$ correspond to vectors in $\Z^{m+n-1}$.

The following equality is key to our arguments. Suppose $n\geq3$. Rearranging columns we rewrite $\lmn$ as follows (and keep the same notation),
$$\lmn=(e_1\mb e_2\cdots\mb e_{m+n-2}\mb B_{m+1}\mb\cdots\mb B_{m+n-2}\mb e_{m+n-1}\mb  B_{m+n-1}).$$
Hence, we have the equality,
$$\lmn=
\begin{pmatrix}
\lmu	& 		&  \\
---	& e_{m+n-1}		& B_{m+n-1} \\
 0	& 				&
\end{pmatrix}.
$$
Indeed, notice that the columns of $\lmu$ correspond to elements of $\Z^{m+n-2}$, which explains the 0 below the dashed line. Hence the equality makes sense. We write $\lmn=(\tlmu\mb e_{m+n-1}\mb B_{m+n-1})$, where
$$\tlmu=
\begin{pmatrix}
\lmu	 \\
---	 \\
 0	 
\end{pmatrix}.
$$

Let $L=(v_1\mb\cdots\mb v_{m+n-1})$ be a maximal submatrix of $\lmn$. We show that $\det L$ is $0$, $1$, or $-1$. We consider four main cases:
\begin{itemize}
\item[(I)] All vectors $v_i$ are taken from $\tlmu$.
\item[(II)] All vectors $v_i$ are taken from $\tlmu$ except one, which is taken from $\{e_{m+n-1}\}\cup B_{m+n-1}$.
\item[(III)] One $v_i$ is $e_{m+n-1}$ and there is at least one $v_j$ taken from $B_{m+n-1}$.
\item[(IV)] All vectors $v_i$ are taken from $\tlmu$ and $B_{m+n-1}$.
\end{itemize}

\noindent In case (I) $\det L=0$ since the last row of $L$ contains only zero entries.
\\

\noindent Case (II). Assume for simplicity of notation that $v_1\in\{e_{m+n-1}\}\cup B_{m+n-1}$ and $v_2,\ldots,v_{m+n-1}$ are taken from $\tlmu$. In this case $L$ has the shape
$$
\begin{pmatrix}
*	&  A \\
1 	& 0	
\end{pmatrix},
$$
where $A$ is a square matrix whose columns are taken from $\lmu$. By induction $\det L$ is $0$, $1$ or $-1$.
\\

\noindent Case (III). Assume that $v_1=e_{m+n-1}$, $v_2,\ldots,v_l\in B_{m+n-1}$ for $2\leq l \leq m$, and $v_{l+1},\ldots,v_{m+n-1}$ are taken from $\tlmu$. Let $2\leq i_2<\cdots<i_l\leq m$ be such that 
\begin{align}
v_2&=-e_1+e_{i_2}+e_{m+n-1},\notag\\
&\vdots\notag\\
v_l&=-e_1+e_{i_l}+e_{m+n-1}.\notag
\end{align}
For $k\in\{2,\ldots,l\}$, replace $v_k$ by $v_k-v_1=-e_1+e_{i_k}$. These operations turn $L$ into the matrix
$$L'=(e_{m+n-1}\mb -e_1+e_{i_2}\mb\cdots\mb-e_1+e_{i_l}\mb v_{l+1}\mb\cdots\mb v_{m+n-1}).$$
\begin{enumerate}
\item Suppose $v_j=e_1$ for some $j\in\{l+1,\ldots,m+n-1\}$. Then colum operations using $v_j$ turn $L'$ into
$$L''=(e_{m+n-1}\mb e_{i_2}\mb\cdots\mb e_{i_l}\mb v_{l+1}\mb\cdots\mb v_{m+n-1}).$$
Notice that $L''$ has the shape
$$
\begin{pmatrix}
0	&  A \\
1 	& 0	
\end{pmatrix},
$$
where $A$ is a square matrix whose columns are taken from $\lmu$. By induction $\det L$ is $0$, $1$ or $-1$.
\item Suppose $v_j=e_{i_k}$ for some $j\in\{l+1,\ldots,m+n-1\}$ and $k\in\{2,\ldots,l\}$. For simplicity of notation assume $v_{l+1}=e_{i_2}$. In $L'$ replace $-e_1+e_{i_2}$ by $-(-e_1+e_{i_2}-v_{l+1})=e_1$. We obtain the matrix
$$L''=(e_{m+n-1}\mb e_1\mb-e_1+e_{i_3}\mb\cdots\mb-e_1+e_{i_l}\mb v_{l+1}\mb\cdots\mb v_{m+n-1}).$$
Then colum operations using $e_1$ turn $L''$ into
$$L'''=(e_{m+n-1}\mb e_1\mb e_{i_3}\cdots\mb e_{i_l}\mb v_{l+1}\cdots\mb v_{m+n-1}).$$
We obtain the desired determinant by induction as in 1.
\end{enumerate}
In view of 1. and 2. we can assume that $v_{l+1},\ldots,v_{m+n-1}$ belong to 
$$(\{e_1,\ldots,e_m\}\setminus\{e_1,e_{i_2},\ldots,e_{i_l}\})\cup\{e_{m+1},\ldots,e_{m+n-2}\}\cup B_{m+1}\cup\cdots\cup B_{m+n-2}.$$
For $j\in\{1,\ldots,n-2\}$ denote
\begin{align}
B^+_{m+j}&=\{-e_1+e_{i_2}+e_{m+j},\ldots,-e_1+e_{i_l}+e_{m+j}\},\notag\\
B^-_{m+j}&=B_{m+j}\setminus B^+_{m+j}.\notag
\end{align}
\begin{itemize}
\item[3.] Suppose $v_i\in\{e_{m+1},\ldots,e_{m+n-2}\}$ for some $i\in\{l+1,\ldots,m+n-1\}$. In $L'$ replace $-e_1+e_{i_k}$ by $(-e_1+e_{i_k})+v_i$. These new vectors belong to $B_{m+j}$ for some $j\in\{1,\ldots,n-2\}$. Hence, these column operations turn $L'$ into a matrix of the form
$$
\begin{pmatrix}
0	&  A \\
1 	& 0	
\end{pmatrix},
$$
where $A$ is a square matrix whose columns are taken from $\lmu$. By induction $\det L$ is $0$, $1$ or $-1$.
\item[4.] Suppose $v_i\in B^+_{m+1}\cup\cdots\cup B^+_{m+n-2}$ for some $i\in\{l+1,\ldots,m+n-1\}$. For simplicity of notation assume $v_{l+1}=-e_1+e_{i_2}+e_{m+j}$, for some $j\in\{1,\ldots,n-2\}$. In $L'$ replace $v_{l+1}$ by $v_{l+1}-(-e_1+e_{i_2})=e_{m+j}$. This brings us back to case 3.
\end{itemize}
In view of 3. and 4. we can assume that $v_{l+1},\ldots,v_{m+n-1}$ belong to 
$$W_0=(\{e_1,\ldots,e_m\}\setminus\{e_1,e_{i_2},\ldots,e_{i_l}\})\cup B^-_{m+1}\cup\cdots\cup B^-_{m+n-2}.$$
Since $2\leq l \leq m$ and $n\geq3$, we have that $\{v_{l+1},\ldots,v_{m+n-1}\}$ is non-empty. Hence $W_0\neq\emptyset$ and so $l<m$. Let $W\subset\R^{m+n-1}$ be the vector space generated by $W_0$. Notice that a basis for $W$ is
$$(\{e_1,e_2,\ldots,e_m\}\setminus\{e_1,e_{i_2},\ldots,e_{i_l}\})\cup \{-e_1+e_{m+1}\}\cup\cdots\cup \{-e_1+e_{m+n-2}\}.$$
Hence $\dim_{\R}W=(m-l)+(n-2)$. Since $|\{v_{l+1},\ldots,v_{m+n-1}\}|=m+n-1-l=(m-1)+(n-1)$, we conclude that $\{v_{l+1},\ldots,v_{m+n-1}\}$ is linearly dependent. Hence $\det L=0$. This concludes the proof of Case (III).
\\

\noindent Case (IV). Assume that $v_2,\ldots,v_l\in B_{m+n-1}$ for $2\leq l \leq m$, and $v_1,v_{l+1},$ $\ldots,v_{m+n-1}$ are taken from $\tlmu$. Let $2\leq i_2<\cdots<i_l\leq m$ be such that 
\begin{align}
v_2&=-e_1+e_{i_2}+e_{m+n-1},\notag\\
&\vdots\notag\\
v_l&=-e_1+e_{i_l}+e_{m+n-1}.\notag
\end{align}
\begin{enumerate}
\item Suppose $v_j=e_{i_k}$ for some $j\in\{1,l+1,\ldots,m+n-1\}$ and $k\in\{2,\ldots,l\}$. For simplicity of notation assume $v_1=e_{i_2}$. In $L$ replace $v_2$ by $v_2-v_1=-e_1+e_{m+n-1}$.  We obtain the matrix
$$L'=(v_1\mb -e_1+e_{m+n-1}\mb v_3\mb\cdots\mb v_l \mb v_{l+1}\mb\cdots\mb v_{m+n-1}).$$
For each $k\in\{3,\ldots,l\}$, replace $v_k$ by $v_k-(-e_1+e_{m+n-1})=e_{i_k}$. These colum operations turn $L'$ into
$$L''=(v_1\mb -e_1+e_{m+n-1}\mb e_{i_3}\cdots\mb e_{i_l}\mb v_{l+1}\cdots\mb v_{m+n-1}).$$
Finally, permute the two first columns to obtain:
$$L'''=(-e_1+e_{m+n-1}\mb\mb e_{i_2}\mb e_{i_3}\cdots\mb e_{i_l}\mb v_{l+1}\cdots\mb v_{m+n-1}).$$
$L'''$ is a matrix of the form
$$
\begin{pmatrix}
*	&  A \\
1 	& 0	
\end{pmatrix},
$$
where $A$ is a square matrix whose columns are taken from $\lmu$. By induction $\det L$ is $0$, $1$ or $-1$.
\end{enumerate}
In view of 1. we can assume that $v_1,v_{l+1},\ldots,v_{m+n-1}$ belong to 
$$(\{e_1,\ldots,e_m\}\setminus\{e_{i_2},\ldots,e_{i_l}\})\cup\{e_{m+1},\ldots,e_{m+n-2}\}\cup B_{m+1}\cup\cdots\cup B_{m+n-2}.$$
For $j\in\{1,\ldots,n-2\}$ denote
\begin{align}
B^+_{m+j}&=\{-e_1+e_{i_2}+e_{m+j},\ldots,-e_1+e_{i_l}+e_{m+j}\},\notag\\
B^-_{m+j}&=B_{m+j}\setminus B^+_{m+j}.\notag
\end{align}
\begin{itemize}
\item[3.] Suppose $v_i\in B^+_{m+1}\cup\cdots\cup B^+_{m+n-2}$ for some $i\in\{1,l+1,\ldots,m+n-1\}$. For simplicity of notation assume $v_1=-e_1+e_{i_2}+e_{m+j}$.
In $L$ replace $v_1$ by $v_1-v_2=e_{m+j}-e_{m+n-1}$. We obtain the following matrix
$$L'=(e_{m+j}-e_{m+n-1}\mb v_2\mb\cdots\mb v_l \mb v_{l+1}\mb\cdots\mb v_{m+n-1}).$$
For each $k\in\{2,\ldots,l\}$ replace $v_k$ by $v_k+(e_{m+j}-e_{m+n-1})=-e_1+e_{i_k}+e_{m+j}$. We obtain the matrix
$$L''=(e_{m+j}-e_{m+n-1}\mb -e_1+e_{i_2}+e_{m+j}\mb\cdots\mb -e_1+e_{i_l}+e_{m+j}\mb v_{l+1}\mb\cdots\mb v_{m+n-1}).$$
$L''$ is a matrix of the form
$$
\begin{pmatrix}
*	&  A \\
-1 	& 0	
\end{pmatrix},
$$
where $A$ is a square matrix whose columns are taken from $\lmu$. We conclude by induction.
\end{itemize}
In view of 3. we can assume that $v_1,v_{l+1},\ldots,v_{m+n-1}$ belong to 
$$W_0=(\{e_1,\ldots,e_m\}\setminus\{e_{i_2},\ldots,e_{i_l}\})\cup \{e_{m+1},\ldots,e_{m+n-2}\}\cup B^-_{m+1}\cup\cdots\cup B^-_{m+n-2}.$$
Let $W\subset\R^{m+n-1}$ be the vector space generated by $W_0$. Notice that a basis for $W$ is
$$(\{e_1,e_2,\ldots,e_m\}\setminus\{e_{i_2},\ldots,e_{i_l}\})\cup \{e_{m+1},\ldots e_{m+n-2}\}.$$
Hence $\dim_{\R}W=m-(l-1)+(n-2)=(m-l)+(n-1)$. Since $|\{v_1,v_{l+1},\ldots,v_{m+n-1}\}|=m+n-1-l+1=(m-1)+n$, we conclude that $\{v_1,v_{l+1},\ldots,v_{m+n-1}\}$ is linearly dependent. Hence $\det L=0$. This concludes the proof of Case (IV).

\end{proof}

\section*{Acknowledgements}

The second author would like to thank for the great hospitality received from ICMC USP and UFSCar, Sao Carlos, during the visit in which this project started.



\begin{thebibliography}{XXX}
\addcontentsline{toc}{section}{\numberline{References}}
\bibitem{ACGH}Arbarello, E., Cornalba, M., Griffiths, P. A., Harris, J.; \textit{Geometry of algebraic curves, Vol. I}, 267, Grundlehren der Mathematischen Wissenschaften, Springer-Verlag, New York, 1985.
\bibitem{At}Atanasov, A., Lopez, C., Perry, A., Proudfoot, N., Thaddeus, M.; \textit{Resolving toric varieties with Nash blow-ups}, Experimental Math. \textbf{20} (2011), no. 3, 288-303.
\bibitem{BV}Bruns, W., Vetter, U.; \textit{Determinantal Rings}, Springer-Verlang, New York, 1998.
\bibitem{Cha}Chachapoyas, N.; \textit{Invariantes de variedades determinantais}, Thesis, ICMC-USP (2014).
\bibitem{ChDG}Ch\'avez-Mart\'inez, E., Duarte, D., Giles Flores, A.; \textit{A higher-order tangent map and a conjecture on the higher Nash blowup of curves}, Math. Z., Vol. 297, (2021), 1767-1791.
\bibitem{CLS}Cox, D., Little, J., Schenck, H.; \textit{Toric Varieties}, Graduate Studies in Mathematics, Volume 124, AMS, 2011.
\bibitem{D1}Duarte, D.; \textit{Nash modification on toric surfaces}, Revista de la Real Academia de Ciencias Exactas, F\'isicas y Naturales, Serie A Matem\'aticas, Vol. 108, No. 1, (2014), pp 153-171.
\bibitem{DG}Duarte, D., Green Tripp, D.; \textit{Nash modification on toric curves}, Singularities, Algebraic Geometry, Commutative Algebra, and Related Topics, Springer Nature Switzerland AG, G.-M. Greuel, L. Narv\'aez Macarro, S. Xamb\'o-Descamps (eds), DOI: 10.1007/978-3-319-96827-8-8, (2018), 191-202.
\bibitem{DJNB}Duarte, D., Jeffries, J., N\'u\~nez-Betancourt, L.; \textit{Nash blowups of toric varieties in prime characteristic}, Collectanea Mathematica, doi:10.1007/s13348-023-00402-y, published online, (2023).
\bibitem{DN1} Duarte, D., N\'u\~nez-Betancourt, L.; \textit{Nash blowups in positive characteristic}, Revista Matem\'atica Iberoamericana, electronically published, DOI: 10.4171/RMI/1278 (2021).
\bibitem{EH}Eagon, J. A., Hochster, M.; \textit{Cohen-Macaulay rings, invariant theory, and the
generic perfection of determinantal loci}, Amer. J. Math. 93 (1971), 1020–1058.
\bibitem{EGZ}Ebeling, W., Gusein-Zade, S. M.; \textit{On the indices of 1-forms on determinantal singularities}, Tr. Mat. Inst. Steklova, 267 (2009) 119–131.
\bibitem{GS1}Gonzalez-Sprinberg, G.; \textit{Eventails en dimension 2 et transform\'{e} de Nash}, Publ. de l'E.N.S., Paris (1977), 1-68.
\bibitem{GS2}Gonzalez-Sprinberg, G.; \textit{R\'{e}solution de Nash des points doubles rationnels}, Ann. Inst. Fourier, Grenoble \textbf{32}, 2 (1982), 111-178.
\bibitem{GS3}Gonzalez-Sprinberg, G.; \textit{On Nash blow-up of orbifolds}, Adv. Studies in Pure Math. 56 (2009), Singularities-Niigata-Toyama 2007, 133-149.
\bibitem{GT}Gonz\'alez Perez, P. D., Teissier, B.; \textit{Toric geometry and the Semple-Nash modification}, Revista de la Real Academia de Ciencias Exactas, F\'isicas y Naturales, Serie A, Matem\'aticas, Vol. 108, Issue 1, (2014), 1-48.
\bibitem{GM}Grigoriev, D., Milman, P.; \textit{Nash resolution for binomial varieties as Euclidean division. A priori termination bound, polynomial complexity in essential dimension 2}, Advances in Mathematics, Vol. 231, (2012), pp 3389-3428.
\bibitem{Hi}Hironaka, H.; \textit{On Nash blowing-up}, Arithmetic and Geometry II, Progr. Math., Vol 36, Birkhauser Boston, Mass., (1983), pp 103-111.
\bibitem{LJ-R}Lejeune-Jalabert, M., Reguera, A.; \textit{The Denef-Loeser series for toric surfaces singularities}, Proceedings of the International Conference on Algebraic Geometry and Singularities (Spanish) (Sevilla, 2001), Rev. Mat. Iberoamericana, Vol. 19, (2003), 581-612.
\bibitem{M2}Grayson, D., Stillman, M.; Macaulay2, a software system for research in algebraic geometry, available at \url{http://www.math.uiuc.edu/Macaulay2/}.
\bibitem{No}Nobile, A.; \textit{Some properties of the Nash blowing-up}, Pacific Journal of Mathematics, \textbf{60}, (1975), 297-305.
\bibitem{O}Oda, T.; \textit{Convex bodies and algebraic geometry}, Ergebnisse der Mathematik und ihrer Grenzgebiete (3), Vol. 15, Springer-Verlag, Berlin, 1988.
\bibitem{Ox}Oxley, J.; \textit{Matroid theory}, Oxford Graduate Texts in Mathematics, 2nd. Edition, Vol. 21, 2011.
\bibitem{R}Rebassoo, V.; \textit{Desingularisation properties of the Nash blowing-up process}, Thesis, University of Washington (1977).
\bibitem{S}Semple, J. G.; \textit{Some investigations in the geometry of curve and surface elements}, Proc. London Math. Soc. (3) \textbf{4} (1954), pp 24-49.
\bibitem{Sp}Spivakovsky, M.; \textit{Sandwiched singularities and desingularisation of surfaces by normalized Nash transformations}, Ann. of Math. (2), Vol. 131, No. 3, (1990), pp 411-491.
\bibitem{St}Sturmfels, B.; \textit{Gr\"obner Bases and Convex Polytopes}, University Lecture Series, Vol. 8, AMS, Providence, RI, 1996.
\end{thebibliography}
\end{document}